\documentclass[12pt]{amsart}

\usepackage{fullpage,mathrsfs,amssymb,color,url,amsmath,mathtools}
\usepackage{tikz-cd}
\usepackage[all, cmtip]{xy}
\definecolor{hot}{RGB}{65,105,225}
\usepackage[pagebackref=true,colorlinks=true, linkcolor=hot ,  citecolor=hot, urlcolor=hot]{hyperref}

\def\sD{\mathscr{D}}
\def\ra{\rightarrow}
\def\bk{\mathbf{k}}
\def\bs{\mathbf{s}}
\def\bf{\mathbf{f}}

\def\ba{\mathbf{a}}

\def\bal{{\boldsymbol{\alpha}}}

\def\bla{{\boldsymbol{\lambda}}}
\def\bC{\mathbb{C}}
\def\bQ{\mathbb{Q}}

\def\bZ{\mathbb{Z}}
\def\bN{\mathbb{N}}

\def\cS{\mathcal{S}}

\def\cM{\mathcal{M}}

\def\al{\alpha}
\def\sA{\mathscr{A}}

\def\sO{\mathscr{O}}

\newcommand{\Ann}{\textup{Ann}}
\newcommand{\supp}{\textup{supp}}

\newcommand{\Mod}{\textup{Mod}}

\newcommand{\gr}{\textup{gr}\,}
\newcommand{\grF }{\textup{gr}^F}

\newcommand{\Spec}{\textup{Spec}\,}

\newcommand{\DR}{\textup{DR}}

\newcommand{\Chr}{\textup{Ch}^{\textup{rel}}}
\newcommand{\be}{\begin{equation} }
\newcommand{\ee}{\end{equation} }

\def\Exp{\textup{Exp}}
\def\Ext{\textup{Ext}}
\def\Rhom{\textup{RHom}}

\theoremstyle{plain}
\newtheorem{theorem}{Theorem}[subsection]
\newtheorem{lemma}[theorem]{Lemma}
\newtheorem{prop}[theorem]{Proposition}

\theoremstyle{definition}
\newtheorem{rmk}[theorem]{Remark}

\newtheorem{definition}[theorem]{Definition}
\newtheorem{defn}[theorem]{Definition}

\title{Zero loci of Bernstein-Sato ideals}
\author{Nero Budur}
\address{KU Leuven, Celestijnenlaan 200B, B-3001 Leuven, Belgium} 
\email{nero.budur@kuleuven.be}
\author{Robin van der Veer}
\address{KU Leuven, Celestijnenlaan 200B, B-3001 Leuven, Belgium} 
\email{robin.vanderveer@kuleuven.be}
\author{Lei Wu}
\address{Department of Mathematics, University of Utah, 155 S. 14000 E, Salt Lake City, UT 84112, USA}
\email{lwu@math.utah.edu}
\author{Peng Zhou}
\address{Institut des Hautes \'Etudes Scientifiques, Le Bois-Marie, 35 Route de Chartres, 91440 Bures-sur-Yvette, France}
\email{pzhou.math@gmail.com}


\numberwithin{equation}{section}

\begin{document}

\begin{abstract} We prove a conjecture of the first author relating the Bernstein-Sato ideal of a finite collection of multivariate polynomials with cohomology support loci of rank one complex local systems. This generalizes a classical theorem of Malgrange and Kashiwara relating the $b$-function of a multivariate polynomial  with the monodromy eigenvalues on the Milnor fibers cohomology.
\end{abstract}

\maketitle

\tableofcontents

\section{Introduction}

\subsection{}\label{subsDef}  Let $F=(f_1,\ldots,f_r):(X,x)\ra (\bC^r,0)$ be the germ of a holomorphic map from a complex manifold $X$. The {\it (local) Bernstein-Sato ideal} of $F$ is the ideal $B_F$ in $\bC[s_1,\ldots,s_r]$ generated by all $b\in \bC[s_1,\ldots,s_r]$  such that in a neighborhood of $x$
\be\label{eqBF}
b\prod_{i=1}^rf_i^{s_i}=P\cdot\prod_{i=1}^rf_i^{s_i+1}
\ee
for some $P\in\sD_{X}[s_1,\ldots,s_r]$, where $\sD_{X}$ is the ring of holomorphic  differential operators. Sabbah \cite{Sab} showed that $B_F$ is not zero.

\subsection{}
If $F=(f_1,\ldots,f_r):X\ra\bC^r$ is a morphism from a smooth complex affine irreducible algebraic variety, the {\it (global) Bernstein-Sato ideal} $B_F$ is defined as the ideal generated by all $b\in \bC[s_1,\ldots,s_r]$ such that (\ref{eqBF}) holds globally with $\sD_X$ replaced by the ring of algebraic  differential operators. The global Bernstein-Sato ideal is the intersection of all the local ones at points $x$ with some $f_i(x)=0$, and there are only finitely many distinct local Bernstein-Sato ideals, see \cite{BMM-c}, \cite{BO}.

\subsection{} It was clear from the beginning that $B_F$ contains some topological information about $F$, e.g. \cite{KK,  Sab, Lo}. However, besides the case $r=1$, it was not clear what  precise topological information is provided by $B_F$. Later, a conjecture  based on computer experiments was formulated in \cite{B-ls} addressing this problem. In this article we prove this conjecture.

\subsection{} Let us recall what happens in the case $r=1$. If $f:X\ra \bC$  is a regular function on a smooth complex affine irreducible algebraic variety, or the germ at $x\in X$ of a holomorphic function on a complex manifold, the monic generator of the Bernstein-Sato ideal of $f$ in $\bC[s]$ is called the {\it Bernstein-Sato polynomial}, or the {\it $b$-function}, of $f$ and it is denoted by $b_f(s)$. The non-triviality of $b_f(s)$ is a classical result of Bernstein \cite{Ber} in the algebraic case, and Bj\"ork \cite{Bjo} in the analytic case. One has the following classical theorem, see \cite{Mal, Kas, Ka}:

\begin{theorem}  Let $f:X\ra \bC$ be a regular function on a smooth complex affine irreducible algebraic variety, or the germ at $x\in X$ of a holomorphic function on a complex manifold, such that $f$ is not invertible. Let $b_f(s)\in\bC[s]$ be the Bernstein-Sato polynomial of $f$. Then:

(i) (Malgrange, Kashiwara)  The set
$$
\{exp(2\pi i\al)\mid \al\text{ is a root of }b_f(s)\}
$$
is the set of monodromy eigenvalues on the nearby cycles complex of $f$.

(ii) (Kashiwara) The roots of $b_f(s)$ are negative rational numbers.

(iii) (Monodromy Theorem) The monodromy eigenvalues on the nearby cycles complex of $f$ are roots of unity.
\end{theorem}

The definition of the nearby cycles complex is recalled in Section \ref{secProof}. In the algebraic case, $b_f(s)$ provides thus an algebraic computation of the monodromy eigenvalues.

\subsection{} We complete in this article the extension of this theorem to a finite collection of functions as follows.
Let $$Z(B_F)\subset\bC^r$$ be the zero locus of the Bernstein-Sato ideal of $F$. Let $\psi_F\bC_X$ be the specialization complex\footnote{This is called ``le complexe d'Alexander" in \cite{Sabb}.} defined by Sabbah \cite{Sabb}; the definition will be recalled in Section \ref{secProof}. This complex is a generalization of the nearby cycles complex to a finite collection of functions, the monodromy action being now given by $r$ simultaneous monodromy actions, one for each function $f_i$. Let $$\cS(F)\subset (\bC^*)^r$$  be the support of this monodromy action on $\psi_F\bC_X$. In the case $r=1$, this is the set of eigenvalues of the monodromy on the nearby cycles complex. The support $\cS(F)$ has a few other topological interpretations, one being in terms of cohomology support loci of rank one local systems, see Section \ref{secProof}. Let $\textup{Exp}:\bC^r\ra (\bC^*)^r$ be the map $\textup{Exp}(\_)=\exp(2\pi i\_)$. 
 

\begin{theorem}\label{thrmE} Let $F=(f_1,\ldots,f_r):X\ra\bC^r$ be a morphism of smooth complex affine irreducible algebraic varieties, or the germ at $x\in X$ of a holomorphic map on a complex manifold, such that not all $f_i$ are invertible. Then:

(i) $\textup{Exp}(Z(B_F))=\cS(F).$ 

(ii) Every irreducible component of $Z(B_F)$ of codimension 1 is a  hyperplane of type $a_1s_1+\ldots+a_rs_r+b=0$ with $a_i\in\bQ_{\ge 0}$ and $b\in\bQ_{>0}$. Every irreducible component of  $Z(B_F)$ of codimension $>1$  can be translated by an element of $\bZ^r$ inside a component of codimension 1.

(iii) $\cS(F)$ is a finite union of torsion-translated complex affine subtori of codimension 1 in $(\bC^*)^r$.

\end{theorem}

Thus in the algebraic case, $B_F$  gives an algebraic computation of $\cS(F)$. 

Part $(i)$ was conjectured in \cite{B-ls}, where one inclusion was also proved, namely that $\textup{Exp}(Z(B_F))$ contains $\cS(F)$. See also \cite[Conjecture 1.4, Remark 2.8]{BLSW}.

Regarding part $(iii)$, Sabbah \cite{Sabb} showed that $\cS(F)$ is included in a finite union of torsion-translated complex affine subtori of codimension 1. Here a complex affine subtorus of $(\bC^*)^r$ means an algebraic subgroup $G\subset (\bC^*)^r$ such that $G\cong (\bC^*)^p$ as algebraic groups for some $0\le p\le r$. In \cite{BW16}, it was proven that every irreducible component of $\cS(F)$ is a torsion-translated subtorus. Finally, part $(iii)$ was proven as stated in \cite{BLSW}. 

The first assertion of part $(ii)$, about the components of codimension one of $Z(B_F)$, is due to Sabbah \cite{Sab} and Gyoja \cite{G}. 

In light of the conjectured equality in part $(i)$, it was therefore expected that part $(iii)$ would hold for $\Exp(Z(B_F))$. This is  equivalent to the second assertion in part $(ii)$, about the smaller-dimensional components of $Z(B_F)$, and it was  
 confirmed unconditionally by Maisonobe \cite[R\'esultat 3]{M}. This result of Maisonobe will play a crucial role in this article.

In this article we complete the proof of  Theorem \ref{thrmE} by proving the other inclusion from part $(i)$:

\begin{theorem}\label{thmconj}  Let  $F$ be as in Theorem \ref{thrmE}. Then $\textup{Exp}(Z(B_F))$ is contained in $\cS(F)$.
\end{theorem}

The proof uses Maisonobe's results from \cite{M} and uses an analog of the Cohen-Macaulay property for modules over the noncommutative ring $\sD_X[s_1,\ldots,s_r]$.

\subsection{} Algorithms for computing Bernstein-Sato ideals are now implemented in many computer algebra systems. The availability of examples where the zero loci of Bernstein-Sato ideals contain irreducible components of codimension $>1$ suggests that this is not a rare phenomenon, see \cite{BO}. The stronger conjecture  that Bernstein-Sato ideals are generated by products of linear polynomials remains open, \cite[Conjecture 1]{B-ls}. This would imply in particular that all irreducible components of $Z(B_F)$ are linear.

\subsection{} In Section \ref{secProof}, we  recall the definition and some properties of the support of the specialization complex. In Section \ref{secFP} we give the proof of Theorem \ref{thmconj}. Section \ref{secApp} is an appendix reviewing basic facts from homological algebra for modules over not-necessarily commutative rings. 

\subsection{Acknowledgement.} We would like to thank L. Ma, P. Maisonobe, C. Sabbah for some discussions, to M. Musta\c{t}\u{a} for drawing our attention to a mistake in the first version of this article, and to the referees for comments that helped improve the article.

The first author was partly supported by the grants STRT/13/005 and Methusalem METH/15/026 from KU Leuven, G097819N and G0F4216N from the Research Foundation - Flanders. The second author is supported by a PhD Fellowship of the Research Foundation - Flanders. The fourth author is supported by the Simons Postdoctoral Fellowship as part of the Simons Collaboration on HMS.

\section{The support of the specialization complex}\label{secProof}
 
\subsection{Notation}   Let $F=(f_1,\ldots,f_r):X\ra\bC^r$ be  a holomorphic map on a complex manifold $X$ of dimension $n>0$.  Let $f=\prod_{i=1}^rf_i$, $D=f^{-1}(0)$, $U=X\setminus D$. Let $i:D\ra X$ be the closed embedding and $j:U\ra X$ the open embedding. We are assuming that not all $f_i$ are invertible, which is equivalent to $D\neq\emptyset$.

We  use the notation $\bs=(s_1,\ldots,s_r)$ and $\bf^\bs=\prod_{i=1}^rf_i^{s_i}$, and in general tuples of numbers will be in bold, e.g. $\mathbf{1}=(1,\ldots, 1)$, $\bal=(\al_1,\ldots,\al_r)$, etc. 


\subsection{Specialization complex}\label{subsDSF}  Let $D^b_c(A_D)$ be the derived category of bounded complexes of $A_D$-modules with constructible cohomology, where $A$ is the affine coordinate ring of $(\bC^*)^r$ and $A_D$ is the constant sheaf of rings on $D$ with stalks $A$. Sabbah \cite{Sabb} defined the {\it specialization complex} $\psi_F\bC_X$ in $D^b_c(A_D)$
by
$$
\psi_F\bC_X=i^{-1}Rj_*R\pi_! (j\circ \pi)^{-1}\bC_X,$$
where $\pi:U\times_{(\bC^*)^r}\bC^r\ra U$ is the first projection from the fibered product obtained from $F_{|U}:U\ra (\bC^*)^r$ and the universal covering map $exp:\bC^r\ra(\bC^*)^r$. 

The {\it support of the specialization complex} $\cS(F)$ is defined as the union over all $i\in \bZ$ and $x\in D$ of the supports in $(\bC^*)^r$ of the cohomology stalks $\mathcal{H}^i(\psi_F\bC_X)_x$ viewed as finitely generated $A$-modules.

If $F$ is only given as the germ at a point $x\in X$ of a holomorphic map, by $\psi_F\bC_X$ we mean the restriction of the specialization complex to a very small open neighboorhood of $x\in X$.

When $r=1$, that is,  in the case of only one holomorphic function $f:X\ra \bC$, the specialization complex  equals the shift by $[-1]$ of Deligne's {\it nearby cycles complex} defined as
$$\psi_f\bC_X=i^{-1}R(j\circ \pi)_* (j\circ \pi)^{-1}\bC_X.$$
The complex numbers in the support $\cS(f)\subset \bC^*$ are called the {\it monodromy eigenvalues} of the nearby cycles complex of $f$.

\subsection{Cohomology support loci}\label{subDTE} It was proven in \cite{B-ls, BLSW} that $\cS(F)$ admits an equivalent definition, without involving derived categories,  as the union of cohomology support loci of rank one local systems on small ball complements along the divisor $D$.  More precisely,
$$
\cS(F)=\{\bla\in (\bC^*)^r\mid H^i(U_x,L_\bla)\ne 0 \text{ for some }x\in D\text{ and }i\in \bZ\},
$$ 
where $U_x$ is the intersection of $U$ with a very small open ball in $X$ centered at $x$, and $L_\bla$ is the rank one $\bC$-local system on $U$ obtained as the pullback via $F:U\ra(\bC^*)^r$ of the rank one local system on $(\bC^*)^r$ with monodromy $\lambda_i$
 around the $i$-th missing coordinate hyperplane. 
 
 If $F$ is only given as the germ at $(X,x)$ of a holomorphic map, $\cS(F)$ is defined as above by replacing $X$ with a very small open neighboorhood of $x$.

For one holomorphic function $f:X\ra\bC$, the support $\cS(f)$ is the union of the sets of eigenvalues of the monodromy acting on cohomologies of the Milnor fibers of $f$ along points of the divisor $f=0$, see \cite[Proposition 1.3]{BW16}.

With this description of $\cS(F)$, the following involutivity property was proven:

\begin{lemma}\label{lemI} (\cite[Theorem 1.2]{BW16}) Let $\bla\in (\bC^*)^r$. Then
$\bla\in \cS(F)$ if and only if $\bla^{-1}\in\cS(F)$.
\end{lemma}

\subsection{Non-simple extension loci}\label{subsSF} An equivalent definition of $\cS(F)$ was found by \cite[\S 1.4]{BLSW} as a locus of rank one local systems on $U$ with non-simple higher direct image in the category of perverse sheaves on $X$ : 
$$
\cS(F)=\left\{\bla\in (\bC^*)^r\mid \frac{Rj_*L_\bla[n]}{j_{!*}L_\bla[n]}\ne 0\right\},
$$
where $L_\bla$ is the rank one local system on $U$ as in \ref{subDTE}. This description is equivalent to
$$
\cS(F)=\left\{\bla\in (\bC^*)^r\mid j_!L_\bla[n]\ra Rj_*L_\bla[n]\text{ is not an isomorphism} \right\},
$$
the map being the natural one.

\subsection{$\sD$-module theoretic interpretation}  Recall that for $\bal\in\bC^r$, 
$$\sD_X[\bs]\bf^{\bs}$$
is the natural left $\sD_X[\bs]$-submodule of the free rank one $\sO_X[\bs,f^{-1}]$-module $\sO_X[\bs,f^{-1}]\cdot\bf^{\bs}$ generated by the symbol $\bf^{\bs}$. For $r=1$, see for example Walther \cite{Wal}.

We denote by $D^b_{rh}(\sD_X)$ the derived category of bounded complexes of regular holonomic $\sD_X$-modules. We denote by $\DR_X:D^b_{rh}(\sD_X)\ra D^b_c(\bC_X)$ the de Rham functor, an equivalence of categories. The following is a particular case of \cite[Theorem 1.3 and Corollary 5.5]{WZ}, see also \cite{BG}:

\begin{theorem}\label{lemERH} Let $F=(f_1,\ldots, f_r):X\ra\bC^r$ be a morphism from a smooth complex algebraic variety. Let $\bal\in\bC^r$ and $\bla=\exp(-2\pi i\bal)$. Let $L_\bla$ be the rank one local system on $U$ defined as in \ref{subDTE}, and let $\cM_\bla=L_\bla\otimes_\bC\sO_U$ the corresponding flat line bundle, so that 
$$
\DR_U(\cM_\bla)=L_\bla[n]
$$
as perverse sheaves on $U$. For every integer $k\gg \Vert\al\Vert$ and $\bk=(k,\ldots,k)\in\bZ^r$, there are natural quasi-isomorphisms in $D^b_{rh}(\sD_X)$
$$
\sD_X[\bs]\bf^{\bs+\bk}\otimes_{\bC[\bs]}\bC_\bal= j_!\cM_\bla,
$$
$$
\sD_X[\bs]\bf^{\bs-\bk}\otimes_{\bC[\bs]}\bC_\bal= j_*\cM_\bla,
$$
where $\bC_\bal$ is the residue field of $\bal$ in $\bC^r$.
\end{theorem}

\begin{prop}\label{propERH} With $F$ as in Theorem \ref{lemERH},
$$
\cS(F) = \Exp \left\{ \bal\in\bC^r \mid \frac{\sD_X[\bs]\bf^{\bs-\bk}}{\sD_X[\bs]\bf^{\bs+\bk}}\otimes_{\bC[\bs]}\bC_\bal \ne 0 \text{ for all }k\gg \Vert\al\Vert \right \}.
$$
\end{prop}
\begin{proof} Applying $\DR_X$ directly to Theorem \ref{lemERH}, one obtains that
$$
\cS(F) = \Exp \left\{ -\bal\in\bC^r \mid \frac{\sD_X[\bs]\bf^{\bs-\bk}}{\sD_X[\bs]\bf^{\bs+\bk}}\otimes^L_{\bC[\bs]}\bC_\bal \ne 0 \text{ in }D^b_{rh}(\sD_X)\text{ for all }k\gg \Vert\al\Vert \right \}
$$
by the interpretation of $\cS(F)$ from \ref{subsSF}.
Since $j_!\cM_\bla\ra j_*\cM_\bla$ is a morphism of holonomic $\sD_X$-modules of same length, the kernel and cokernel must simultaneously vanish or not. Thus, we can replace the derived tensor product with the usual tensor product. We then can replace $-\bal$ with $\bal$ by Lemma \ref{lemI}. 
\end{proof}

For related work in a particular case, see \cite{Bath}.

\begin{rmk} Note that Theorem \ref{lemERH} is stated in the algebraic case only. However, the proof from \cite{BG,WZ} extends to the   case when $X$ is a complex manifold by replacing $j_!\cM_\bla$, $j_*\cM_\bla$ with $\cM(!D)$, $\cM(*D)$, respectively, where $\cM$ is the analytic $\sD_X$-module $\sD_X\cdot \bf^\bal$ whose restriction to $U$ is $\cM_\bla$. Hence the last proposition also holds in the analytic case.
\end{rmk}

Since the tensor product is a right exact functor, as a consequence one has the following corollary which also follows from the proof of \cite[Proposition 1.7]{B-ls}:

\begin{prop}\label{propBS} 
If $\bal$ is in $\bC^r$ and 
$$\frac{\sD_X[\bs]\bf^{\bs}}{\sD_X[\bs]\bf^{\bs+\mathbf{1}}}\otimes_{\bC[\bs]}\bC_\bal \ne 0,$$
then $\Exp(\bal)$ is in $\cS(F)$.
\end{prop}

This proposition can be interpreted as to say that the difficulty in proving Theorem \ref{thmconj} is the lack of a Nakayama Lemma for the non-finitely generated $\bC[\bs]$-module $\sD_X[\bs]\bf^\bs/\sD_X[\bs]\bf^{\bs+\mathbf{1}}$.

\section{Relative holonomic modules}\label{secFP}

 In this section we will provide necessary conditions for modules over $\sD_X[\bs]$ to obey an analog of Nakayama Lemma, and we will see that $\sD_X[\bs]\bf^\bs/\sD_X[\bs]\bf^{\bs+\mathbf{1}}$ satisfies these conditions at least generically. Using Maisonobe's results \cite{M}, this will prove Theorem \ref{thmconj}.

\subsection{} For simplicity, we will assume from now that we are in the algebraic case, namely, $X$ is a smooth complex affine irreducible algebraic variety. We will treat the analytic case at the end. 

We define an increasing filtration on the ring $\sD_X$ by setting $F_i\sD_X$ to consist of all operators of order at most $i$, that is, in local coordinates $(x_1,\ldots,x_n)$ on $X$, the order of $x_i$ is zero and the order of $\partial/\partial {x_i}$ is one.

We let $R$ be a  regular commutative finitely generated $\bC$-algebra integral domain. We write
\[\sA_R=\sD_X\otimes_{\bC}R,\]
and if $R=\bC[\bs]$ we write
\[\sA=\sA_{\bC[\bs]}=\sD_X[\bs].\]
The order filtration on  $\sD_X$ induces the {\it relative filtration} on $\sA_R$ by 
\[F_i\sA_R=F_i\sD_X\otimes_\bC R.\]
The associated graded ring $$\gr\sA_R=\gr \sD_X\otimes_\bC R$$ is a regular commutative finitely generated $\bC$-algebra integral domain, and it corresponds to the structure sheaf of $T^*X\times \Spec R$, where $T^*X$ is the cotangent bundle of $X$. Thus $\sA_R$ is an Auslander regular ring by Theorem \ref{thrmGAAR}. Moreover, the homological dimension is equal to the Krull dimension of $\gr\sA_R$,
$$
\text{gl.dim}(\sA_R) = 2n+\dim(R),
$$
by Proposition \ref{eqGldim}, Proposition \ref{propJgr}, and Proposition \ref{propCo}.

\subsection{}  Let $ N$ be a left (or right) $\sA_R$-module. A {\it good filtration $F$ on  $N$ over $R$} is an exhaustive filtration compatible with the relative filtration on $\sA_R$ such that the associated graded module $\gr  N$ is finitely generated over $\gr\sA_R$, cf. \ref{subFlt}. If $ N$ is finitely generated over $\sA_R$, then good filtrations over $R$ exist on $N$.  We define the {\it relative characteristic variety of $ N$ over $R$} to be the support of $\gr  N$ inside $T^*X\times \Spec R$, denoted by $$\Chr( N).$$ 
Equivalently, $\Chr(N)$ is defined by the radical of the annihilator ideal of $\gr N$ in $\gr\sA_R$. The relative characteristic variety $\Chr( N)$ and the multiplicities $m_\mathfrak{p}(N)$ of $\gr  N$ at generic points $\mathfrak{p}$ of the irreducible components of $\Chr( N)$ do not depend on the choice of a good filtration for $ N$, by \ref{lem:relCh}.

\begin{rmk}\label{rmkF} 
The good filtration $F$ on $N$ localizes, that is, if $S$ is a multiplicatively closed subset of $R$, then 
$$F_i (S^{-1} N)=S^{-1}F_i  N$$
form a good filtration of $S^{-1}N$ over $S^{-1}R$, and hence 
\[ \gr (S^{-1} N)\simeq S^{-1}\gr  N.\] 
\end{rmk}

 For a finitely generated $\sA_R$-module $N$, we will denote by $j_{\sA_R}(N)$, or simply $j(N)$, the grade number of $N$ defined as in \ref{subGrade}.

\begin{lemma}\label{thm:relCh}
Suppose that $ N$ is a finitely generated $\sA_R$-module. Then:
\begin{itemize}
\item[(1)] $j( N)+\dim(\Chr( N))=2n+\dim(R);$
\item[(2)] if $$0\to  N'\to  N\to  N''\to 0$$ is a short exact sequence of finitely generated $\sA_R$-modules, then 
\[\Chr( N)=\Chr( N')\cup\Chr( N'')\]
and if $\mathfrak{p}$ is the generic point of an irreducible component of $\Chr( N)$ then \[m_\mathfrak{p}( N)=m_\mathfrak{p}( N')+m_\mathfrak{p}( N'').\]
\end{itemize}
\end{lemma}
\begin{proof}
Propositions \ref{propJgr} and \ref{propCo} give (1). Proposition \ref{lem:relCh} gives (2). 
\end{proof}

Note that the lemma does not require, nor does it imply, that $\Chr(N)$ is equidimensional.

\begin{defn}\label{defRHL}
We say that a finitely generated $\sA_R$-module $ N$ is {\it relative holonomic over $R$} if its relative characteristic variety over $R$ is a finite union
\[\Chr( N) =\bigcup_{w}\Lambda_w \times S_w\]
where $\Lambda_w$ are irreducible conic Lagrangian subvarieties in $T^*X$ and $S_w$ are algebraic irreducible subvarieties of $\Spec R$. 
\end{defn}

\begin{lemma}\label{lm:relhol}
Suppose that $ N$ is relative holonomic over $R$. Then:
\begin{itemize}
\item[(1)] every nonzero subquotient of $ N$ is  relative holonomic over $R$;
\item[(2)] if $\Ext^j_{\sA_R}( N,\sA_R)\not=0$ for some integer $j$, then $\Ext^j_{\sA_R}( N,\sA_R)$ is  relative holonomic (as a right $\sA_R$-module if $ N$ is a left $\sA_R$-module and vice versa), and 
\[ \Chr(\Ext^j_{\sA_R}( N,\sA_R))\subset \Chr( N).\]
\end{itemize}
\end{lemma}
\begin{proof}
By Proposition \ref{propExGr}, there exist good filtrations on $N$ and $\Ext^j_{\sA_R}( N,\sA_R)$ such that $\gr(\Ext^j_{\sA_R}( N,\sA_R))$ is a subquotient of $\Ext^j_{\gr{\sA_R}}(\gr N,\gr \sA_R)$. It follows that
\[ \Chr(\Ext^j_{\sA_R}( N,\sA_R))\subset \Chr( N).\]
Then part (2) follows from Proposition \ref{propMaj}. Part (1) is proved similarly, using Lemma \ref{thm:relCh} (2).
\end{proof}

The following is a straight-forward generalization of the algebraic case of \cite[Proposition 8]{M} where one replaces $\bC[\bs]$ by $R$:

\begin{prop}\label{propMaj} If $N$ is a finitely generated module over $\sA_R$ such that $\Chr(N)$ is contained in $\Lambda\times \Spec R$ for some conic Lagrangian, not necessarily irreducible, subvariety $\Lambda$ of $T^*X$, then $N$ is relative holonomic over $R$.
\end{prop}
\begin{proof} The Poisson bracket on $\gr \sA_R$ is the $R$-linear extension of the Poisson bracket on $\gr\sD_X$. Let $J$ be the radical ideal of the annihilator in $\gr \sA_R$ of $\gr N$. By Gabber's Theorem \cite[A.III 3.25]{Bj}, $J$ is involutive with respect to the Poisson bracket on $\gr\sA_R$, that is, $\{J,J\}\subset J$. Let $\mathfrak{m}$ be a maximal ideal in $R$ corresponding to a point $q$ in the image of $\Chr(N)$ under the second projection $$p_2:T^*X\times\Spec R\ra\Spec R.$$ By $R$-linearity of the Poisson bracket, it follows that $J+\mathfrak{m}\cdot \sA_R$ is involutive. Therefore the image $\bar J$ of $J$ in the ring $\gr \sA_R\otimes_R R/\mathfrak{m}\simeq \gr\sD_X$ is involutive under the Poisson bracket on $\gr\sD_X$. If this ideal would be radical, we could conclude that all the irreducible components of the fiber $\Chr(N)\cap p_2^{-1}(q)$ have dimension at least $\dim X$. Note however that the same assertions on involutivity are true for the associated sheaves since the Poisson bracket on a $\bC$-algebra induces a canonical Poisson bracket on the localization of the algebra with respect to any multiplicatively closed subset, cf. \cite[Lemma 1.3]{Kal}. Thus, restricting to an open subset of $\Chr(N)$ where the second projection $p_2$ has smooth reduced fibers, and assuming $q=p_2(y)$ for a point $y$ in this open subset, the involutivity implies that   $\dim_y(\Chr(N)\cap p_2^{-1}(q))\ge \dim X$. By the upper-semicontinuity on $\Chr(N)$ of the function $y\mapsto \dim_y (\Chr(N)\cap p_2^{-1}(p_2(y)))$, every irreducible component of a non-empty fiber  $\Chr(N)\cap p_2^{-1}(q)$ has dimension $\ge \dim X$. (So far, this is an elaborate adaptation of proof of the algebraic case of \cite[Proposition 5]{M} to the case when $\bC[\bs]$ is replaced by $R$.)

Since $\Lambda$ is equidimensional with $\dim \Lambda=\dim X$, and $\Lambda$ contains every non-empty fiber $\Chr(N)\cap p_2^{-1}(q)$, it follows that  $\Chr(N)\cap p_2^{-1}(q)$ is a finite union of some of the irreducible conic Lagrangian subvarieties $\Lambda_w$ of $T^*X$ which are irreducible components of $\Lambda$. Define $S_w$ to be the subset of closed points $q$ in $\Spec R$ such that $\Lambda_w$ is an irreducible component of $\Chr(N)\cap p_2^{-1}(q)$. Then $\Chr(N)=\cup_w (\Lambda_w\times S_w)$. Moreover, setting $\lambda_w$ to be a general point of $\Lambda_w$, 
$$
\{\lambda_w\}\times S_w=\Chr(N)\cap p_1^{-1}(\lambda_w),
$$
where $p_1:T^*X\times\Spec R\ra T^*X$ is the first projection. Since the right-hand side is defined in $\Spec R$ by finitely many algebraic regular functions, $S_w$ is Zariski closed in $\Spec R$. It follows that $\Chr (N)$ is relative holonomic over $R$. \end{proof}

\subsection{} Recall from \ref{subGrade} the definition of pure modules over $\sA_R$. Examples of pure modules are given by the following.

\begin{defn}\label{defnImp}
We say that a nonzero finitely generated $\sA_R$-module $ N$ is {\it Cohen-Macaulay}, or more precisely {\it $j$-Cohen-Macaulay}, if for some $j\ge 0$
\[\Ext^k_{\sA_R}( N,\sA_R)=0 \textup{ if }k\not=j.\]

\end{defn}

\begin{rmk}\label{remdefnImp}
If $N$ is a Cohen-Macaulay $\sA_R$-module, then:
\begin{enumerate}
\item $N$ is $j$-pure (see Definition \ref{defjpure}), by Lemma \ref{lemEpure} (2);
\item $\Chr(N)$ is equidimensional of codimension $j$, by Propositions \ref{propPGF}, \ref{propJgr}, and \ref{propCo}.
\end{enumerate}
\end{rmk}

\begin{lemma}\label{cor:maxcm}
If $ N$ is relative holonomic over $R$ and $j( N)=n+\dim( R)$, then it is $(n+\dim( R))$-Cohen-Macaulay.
\end{lemma}
\begin{proof} The condition on $j(N)$ implies that $N\neq 0$ by Lemma \ref{thm:relCh} (1). If $\Ext^k_{\sA_R}( N,\sA_R)\not=0$ for some $k>n+\dim(\Spec R)$, then by Lemma \ref{lm:relhol} (2), $\Ext^k_{\sA_R}( N,\sA_R)$ is relative holonomic. Hence 
$\dim (\Chr(\Ext^k_{\sA_R}( N,\sA_R)))\ge n.$
Since $\sA_R$ is an Auslander regular ring, $j(\Ext^k_{\sA_R}( N,\sA_R))\ge k$. This contradicts Lemma \ref{thm:relCh} (1). 
\end{proof}

\subsection{}
For a finitely generated $\sA_R$-module $ N$, since $ N$ is also an $R$-module, we write 
\[B_{ N}=\Ann_R( N)\]
and denote by $Z(B_N)$ the reduced subvariety in $\Spec R$ defined by the radical ideal of $B_{ N}$. Since in general $ N$ is not finitely generated over $R$, it is a priori not clear that $Z(B_{ N})$ is  the $R$-module support of $ N$, $\supp_R( N)$, consisting of closed points with maximal ideal $\mathfrak{m}\subset R$ such that the localization $N_\mathfrak{m}\ne 0$.

\begin{lemma}\label{lm:Rsupp}
If $ N$ is relative holonomic over $R$, then 
\[Z(B_{ N})=p_2(\Chr( N)),\]
where $p_2\colon T^*X\times \Spec R\to \Spec R$ the natural projection. In particular, 
\[Z(B_{ N})= \supp_{R}( N).\]
\end{lemma}
\begin{proof}
For $R=\bC[\bs]$ and in the analytic setting, this is \cite[Proposition 9]{M}, whose proof can be easily adapted to our case. Since $N$ is relative holonomic, $p_2(\Chr (N))$ is closed. Since the contraction of a radical ideal is a radical ideal, the ideal defining $p_2(\Chr (N))$ is $R\cap \sqrt{\Ann_{\gr \sA_R} (\gr N)}$. Hence the first assertion is equivalent to
$$
R\cap \sqrt{\Ann_{\gr \sA_R} (\gr N)} = \sqrt {\Ann_R(N)},
$$
where $R$ is viewed as a $\bC$-subalgebra of $\gr\sA_R=\gr\sD_X\otimes_\bC R$ via the map $a\mapsto 1\otimes a$ for $a$ in $R$. Let $b$ be in $R$. If $b^k N=0$ for some $k\ge 1$, then $b^k(\gr N)=0$ as well. Conversely, if $b^k(\gr N) =0$ for some $k\ge 1$, then $b^k(F_iN)\subset F_{i-1}N$ for all $i$. Since $\gr N$ is finitely generated over $\gr \sA_R$, the filtration $F$ on $N$ is bounded from below. Then by induction applied to the short exact sequence
$$
0\ra F_{i-1}N\ra F_iN\ra \grF_i N\ra 0,
$$
it follows that for each $i$ there exist a multiple $k_i$ of $k$ such that $b^{k_i}(F_iN)=0$, and $k_i$ form an increasing sequence. Fix a finite set of generators of $N$ over $\sA_R$. Since $F$ is exhaustive, there exists an index $j$ such that all the generators are contained in $F_{j}N$. Then $b^{k_j}N=0$. 

We proved thus the first claim, or equivalently, that $Z(B_N)=\supp_R(\gr N)$. Hence the second assertion follows from the equality $$\supp_R(\gr N)=\supp_R(N)$$ which is proved as follows. If $\mathfrak{m}$ is a maximal ideal in $R$ such that $(\grF_iN)_\mathfrak{m}\ne 0$ for some $i$, then $(F_iN)_\mathfrak{m}\ne 0$ since localization is an exact functor. Then, again by exactness, $N_\mathfrak{m}\ne 0$ since $F_iN$ injects into $N$. Thus $\supp_R(\gr N)$ is a subset of $\supp_R(N)$. Conversely, if $N_\mathfrak{m}\ne 0$, take $i$ to be the minimum integer with the property that $(F_iN)_\mathfrak{m}\ne 0$ but $(F_{i-1}N)_\mathfrak{m}=0$. Then $(\grF_i N)_\mathfrak{m}\ne 0$.
\end{proof}

\begin{lemma}\label{lm:minj}
Suppose that $ N$ is relative holonomic over $R$ and $(n+l)$-pure for some $0\le l\le \dim(R)$. If $b$ is an element of $R$ not contained in any minimal prime ideal containing $B_N$, then the morphisms given by multiplication by $b$ 
\[ N\xrightarrow{b} N\]
and 
\[\Ext^{n+l}_{\sA_R}( N,\sA_R)\xrightarrow{b}\Ext^{n+l}_{\sA_R}( N,\sA_R)\]
are injective.  Furthermore, there exists a good filtration of $ N$ over $R$ so that 
\[\gr  N\xrightarrow{b}\gr  N\]
is also injective.
\end{lemma}
\begin{proof}

We first prove that $ N\xrightarrow{b} N$ is injective.
If on the contrary its kernel $K\not=0$, then by Lemma \ref{thm:relCh} (2)
\[\Chr(K)\subset \Chr( N).\]
By purity, we know that $j(K)=j( N)=n+l$. Thanks to Lemma \ref{thm:relCh} (1), 
\[\dim(\Chr(K))=\dim(\Chr( N)).\]
By Proposition \ref{propPGF}, we can choose good filtrations on $K$ and $ N$ so that both $\gr K$ and $\gr  N$ are $(n+l)$-pure over $\gr\sA_R$. Hence $\Chr(K)$ and $\Chr( N)$ are equidimensional of dimension $n+\dim(R)-l$, by Propositions \ref{propJgr} and \ref{propCo}. In particular, $\Chr(K)$ is a union of some irreducible components of $\Chr(N)$.

By the relative holonomicity of $N$, the irreducible components of $\Chr(N)$ are $\Lambda_i\times Z_i$ with $i$ in some finite index set $I$, for some conic irreducible Lagrangian subvarieties $\Lambda_i\subset T^*X$ and some irreducible closed subsets $Z_i\subset\Spec R$. The equidimensionality of $\Chr(N)$ implies that $\dim Z_i=\dim(R)-l$.

 By Lemma \ref{lm:Rsupp}, $Z(B_N)=\cup_{i\in I}Z_i$, and the assumption on $b$ is that $(b=0)$ does not contain any irreducible component of $Z(B_N)$, where by $(b=0)$ we mean the reduced closed subset of $\Spec R$ defined by the radical ideal of $b$. We hence have 
\[\Chr(K)\not\subset T^*X\times(b=0),\]
 However, since $b$ annihilates $K$, $\Chr(K)\subset T^*X\times(b=0)$, which is a contradiction. 

Similarly, since $\gr  N$ is $(n+l)$-pure over $\gr\sA_R$, we can run the above argument by replacing $\Chr(K)$ with the support of the kernel of the map
\[\gr  N\xrightarrow{b} \gr  N\]
to obtain the injectivity of the latter. 

By Lemma \ref{lm:relhol} (2), $\Ext^{n+l}_{\sA_R}( N,\sA_R)$ is relative holonomic and 
\[\Chr(\Ext^{n+l}_{\sA_R}( N,\sA_R))\subset \Chr( N).\]
Since $\Ext^{n+l}_{\sA_R}( N,\sA_R)$ is always $(n+l)$-pure, cf. Lemma \ref{lemEpure} (1), by a similar argument we conclude that 
\[\Ext^{n+l}_{\sA_R}( N,\sA_R)\xrightarrow{b}\Ext^{n+l}_{\sA_R}( N,\sA_R)\]
is also injective.
\end{proof}

The following is the key technical result of the article. For simplicity, we take $\Spec R$ to be an open set of $\bC^r$, the only case we need for the proof of the main result.

\begin{prop}\label{prop:maincm}
Let $\Spec R$ be a nonempty open subset of $\bC^r$. Let $ N$ be an $\sA_R$-module that is relative holonomic over $R$ and $(n+l)$-Cohen-Macaulay over $\sA_R$ for some $0\le l\le r$. Then 
\[\bal\in Z(B_{ N})\; \text{\it  if and only if }  N\otimes_{R}\bC_\bal\not=0,\]
where $\bC_\bal$ is the residue field of the closed point $\bal\in \Spec R$.
\end{prop}

\begin{proof}
We first assume $ N\otimes_R \bC_\bal\not=0$. Then $ N\otimes_R R_\mathfrak{m}\not=0$, where $\mathfrak{m}\subset R$ is the maximal ideal of $\bal$ and $R_\mathfrak{m}$ is the localization of $R$ at $\mathfrak{m}$.  Then $\bal$ belongs to $\supp_R( N)=Z(B_N)$, by Lemma \ref{lm:Rsupp}.

Conversely, we fix a point $\bal$ in $Z(B_N)$. Since $ N$ is $(n+l)$-Cohen-Macaulay, it is in particular $(n+l)$-pure as a module over $\sA_R$. By Proposition \ref{propPGF}, we then can choose a good filtration $F$ on $ N$ so that $\gr  N$ is also pure over $\gr\sA_R$. Hence $\Chr( N)$ is purely of dimension $n+r-l$. By relative holonomicity and Lemma \ref{lm:Rsupp}, $Z(B_N)$ is also purely of dimension $r-l$. 

Let us consider the case when $l<r$. We then can choose a linear polynomial $b\in \bC[\bs]$ so that $(b=0)$ contains $\bal$, but does not contain any of the irreducible components of $Z(B_N)$. By Lemma \ref{lm:minj}, the morphisms given by multiplication by $b$
\[ N\xrightarrow{b} N
\textup{ and }
\gr  N\xrightarrow{b} \gr  N\]
are both injective, the good filtration from Lemma \ref{lm:minj} being constructed in the same way. Thus for every $i$  the vertical maps are injective in the diagram
$$
\xymatrix{
0 \ar[r] & F_{i-1}N \ar[r] \ar[d]^b& F_iN\ar[r] \ar[d]^b & F_i N/F_{i-1}N \ar[r] \ar[d]^b& 0\\
0 \ar[r] & F_{i-1}N \ar[r] & F_iN \ar[r]& F_i N/F_{i-1}N \ar[r]& 0
}
$$
and hence by the snake lemma we get an exact sequence
\be\label{eqSE1}
0\ra F_{i-1}N\otimes_R R/(b) \ra F_iN\otimes_R R/(b) \ra \grF_i N\otimes_R R/(b)\ra 0.
\ee

Note that $b$ is also injective on $N/F_iN$. Indeed, if not, then there exists some $\nu\in F_jN$ with $j>i$, $\nu\not\in F_{j-1}N$, and $b\nu\in F_iN$.  But then $b$ must annihilate the class of $\nu$ in $\grF_j N$, which contradicts the injectivity of $b$ on $\gr N$. Running a similar snake lemma as above after applying the multiplication by $b$ on the short exact sequence 
$$
0\ra F_iN\ra N\ra N/F_iN\ra 0,
$$
we obtain a short exact sequence
\be\label{eqSE2}
0\ra F_iN\otimes_R R/(b)\ra N\otimes_R R/(b)\ra (N/F_iN)\otimes_R R/(b)\ra 0
\ee
The injectivity from (\ref{eqSE1}) and (\ref{eqSE2}) implies that the induced filtration on $N\otimes_R R/(b)$,
$$
F_i(N\otimes_R R/(b)) =\text{im} (F_iN\ra N\otimes_R R/(b)) \simeq F_iN/(F_iN\cap bN),
$$
is the filtration by $$F_iN\otimes_R R/(b)\simeq F_iN/bF_iN,$$ and the surjectivity from (\ref{eqSE1}) then implies
\be\label{eqNRT}\gr ( N\otimes_R R/(b))\simeq \gr  N\otimes_R R/(b).\ee


By Lemma \ref{lm:Rsupp}, $p_2^{-1}(\bal)$ intersects non-trivially the support of $\gr  N$, hence the same is true for $p_2^{-1}(b=0)$. By Nakayama's Lemma for the finitely generated module $\gr N$ over $\gr \sA_R$, we hence have 
$$
0\neq \frac{\gr N}{b \cdot \gr N} \simeq\gr  N\otimes_R R/(b).
$$
Together with the isomorphism (\ref{eqNRT}), this implies that $ N\otimes_R R/(b)\not=0$. Since $ N\otimes_R R/(b)$ is also a finitely generated $\sA_{R/(b)}$-module and $\gr ( N\otimes_R R/(b))$ is a finitely generated $\gr \sA_{R/(b)}$-module, we further conclude from (\ref{eqNRT}) that the relative characteristic variety over $R/(b)$
\begin{equation}\label{eq:pbd}
\Chr( N\otimes_R R/(b))=(\text{id}_{T^*X}\times\Delta)^{-1}(\Chr( N)),
\end{equation}
where $\Delta\colon \Spec R/(b)\hookrightarrow \Spec R$ is the closed embedding. Hence $ N\otimes_R R/(b)$ is relative holonomic over $R/(b)$. By Lemma \ref{lm:Rsupp}, we further have 
\[Z(B_{ N\otimes R/(b)})=\Delta^{-1}(Z(B_N)).\]
 In particular,
\[\Delta^{-1}(\bal)\in Z(B_{ N\otimes R/(b)}).\]
Since 
\[ N\otimes_{R}\bC_\bal\simeq  N\otimes_R R/(b)\otimes_{R/(b)}\bC_{\Delta^{-1}(\bal)},\]
where $\bC_{\Delta^{-1}(\bal)}$ is the residue field of $\Delta^{-1}(\bal)\in \Spec R/(b)$,
our strategy will be to prove
\[ N\otimes_R \bC_{\bal}\not=0\]
by repeatedly replacing $ N$ by $ N\otimes_R R/(b)$ and $R$ by $R/(b)$.

To make this work, we need first to prove that $ N\otimes_R R/(b)$ remains Cohen-Macaulay over $\sA_{R/(b)}$. By taking a free resolution of $ N$, one can see that 
\be\label{eqRHQI}
\Rhom_{\sA_R}( N, \sA_R)\otimes_{\sA_R}^L \sA_{R/(b)}\simeq  \Rhom_{\sA_{R/(b)}}( N\otimes^L_R R/(b), \sA_{R/(b)})
\ee
in the derived category of right $\sA_{R/(b)}$-modules. Since the multiplication by $b$ is injective on $ N$, we further have 
\be\label{eqQIRH}
\Rhom_{\sA_{R/(b)}}( N\otimes^L_R R/(b), \sA_{R/(b)}) \simeq \Rhom_{\sA_{R/(b)}}( N\otimes_R R/(b), \sA_{R/(b)}).
\ee
We will use the Grothendieck spectral sequence associated with the left-hand side of (\ref{eqRHQI}) to compute the $\Ext$ modules from the right-hand side of (\ref{eqQIRH}). Let us assume without harm that $N$ is a left $\sA_R$-module. Then  viewing $\text{Hom}_{\sA_R}(\_\,, \sA_R)$ as a covariant right-exact functor on the opposite category of the category of left $\sA_R$-modules, the composition of the two derived functors $\Rhom_{\sA_R}(\_\,, \sA_R)$ and $(\_\,) \otimes_{\sA_R}^L \sA_{R/(b)}$ gives us a convergent first quadrant homology spectral sequence
\[E^2_{p,q}=\textup{Tor}_{p}^{\sA_R}(\Ext^q_{\sA_R}( N,\sA_R), \sA_{R/(b)})\Rightarrow \Ext^{-p+q}_{\sA_{R/(b)}}( N\otimes_R R/(b),\sA_{R/(b)}),\]
by \cite[Corollary 5.8.4]{HAWei}. Note that the conditions from {\it loc. cit.} are  satisfied in our case, since a projective object in the opposite category of the category of left $\sA_R$-modules is an injective left $\sA_R$-module $I$, and thus $\text{Hom}_{\sA_R}(I, \sA_R)$ is a projective right $\sA_R$-module, and so acyclic for the left exact functor  $(\_\,) \otimes_{\sA_R} \sA_{R/(b)}$.

Since $ N$ is $(n+l)$-Cohen-Macaulay over $\sA_R$, 
\[\Ext^q_{\sA_R}( N,\sA_R)=0 \textup{ for } q\not=n+l.\]
Then 
\[\textup{Tor}_{p}^{\sA_R}(\Ext^{n+l}_{\sA_R}( N,\sA_R), \sA_{R/(b)})=0 \textup{ for } p\not= 0\]
thanks to Lemma \ref{lm:minj},
since the complex
$
\sA_R\stackrel{b}{\ra} \sA_R
$
is a resolution of $\sA_{R/b}$. Therefore the above spectral sequence degenerates at $E^2$,
\[\Ext^{q}_{\sA_{R/(b)}}( N\otimes_R R/(b),\sA_{R/(b)})=0 \textup{ for } q\not= n+l,\]
and
\[\Ext^{n+l}_{\sA_{R/(b)}}( N\otimes_R R/(b),\sA_{R/(b)})\simeq \Ext^{n+l}_{\sA_R}( N,\sA_R)\otimes_{\sA_R} \sA_{R/(b)} \simeq \Ext^{n+l}_{\sA_R}( N,\sA_R)\otimes_R R/(b).\]
As a consequence, $ N\otimes_R R/(b)$ is $(n+l)$-Cohen-Macaulay over $\sA_{R/(b)}$. 

Since $b$ is linear, $\bC^{r-1}\simeq \Spec \bC[\bs]/(b)$, and the latter contains $\Spec R/(b)$ an open subset. We then repeatedly replace $R$ by $R/(b)$, $ N$ by $ N\otimes_R R/(b)$, and $\bal$ by $\Delta^{-1}(\bal)$. Each time $r$ drops by 1, $l$ stays unchanged, and $ N$ remains nonzero, relative holonomic, and $(n+l)$-Cohen-Macaulay. This reduces us  to the case $l=r$. 

If $0=l=r$, the claim is trivially true.

We now assume $0<l=r$. Since $ N$ is now relative holonomic and $(n+r)$-Cohen-Macaulay, hence $(n+r)$-pure, we have
\[\Chr( N)=\sum_w \Lambda_w\times \{p_w\}\]
where $p_w$ are  points in $\bC^r$. Hence $Z(B_{ N})$ is a finite union of points in $\Spec R$. Counting multiplicities, by Lemma \ref{thm:relCh} (2) we see that $ N$ is of finite length. 

 We now fix a linear polynomial $b\in \bC[\bs]$ with $b(\bal)=0$ but not vanishing at the other points of $Z(B_N)$. We then have an exact sequence 
\[0\to K\ra  N \xrightarrow{b}  N \ra  N\otimes_R R/(b)\to 0\]
where $K$ is the kernel. We claim that $K\neq 0$. To see this, chose a polynomial $c\in\bC[\bs]$ not vanishing at $\bal$ but vanishing at all other points of $Z(B_N)$. Then by Nullstellensatz, there is $m>0$ the smallest power such that $(bc)^m$ is in $B_N$. On the other hand, $c^m$ is not in $B_N$. Taking $p\ge 1$ to be the smallest with  $b^pc^m\in B_N$,  we see that there exists $\nu$ in $N$, such that $b^{p-1}c^m\nu$ is a nonzero element of $K$.

Since $K\not=0$  and since endomorphisms of modules of finite length are isomorphisms if and only if they are surjective, we have $N\otimes_R R/(b)\not=0$.  By  Lemma \ref{lm:relhol} (1), $N\otimes_R R/(b)$ is relative holonomic over $R$, and by Lemma \ref{thm:relCh} (2), every irreducible component of its relative characteristic variety over $R$ is one of the components $\Lambda_w\times\{p_w\}$ of $\Chr(N)$. Since $b$ annihilates $N\otimes_R R/(b)$, only the components with $b(p_w)=0$, and hence with $p_w=\bal$, appear. We conclude that  $N\otimes_R R/(b)$ is also relative holonomic over $R/(b)$. By Lemma \ref{thm:relCh} (1), we have $j_{\sA_{R/(b)}}( N\otimes_R R/(b))=n+r-1$. Then by Lemma \ref{cor:maxcm}, $ N\otimes_R R/(b)$ is $(n+r-1)$-Cohen-Macaulay over $\sA_{R/(b)}$.

We therefore can replace $ N$ by $ N\otimes_R R/(b)$, $R$ by $R/(b)$, and assume that $\Chr(N)=\cup_w \Lambda_w\times \{\bal\}$ for some irreducible conic Lagrangian subvarieties $\Lambda_w$ of $T^*X$. Repeating this process, each time $r$ drops by 1, $ N$ remains nonzero, relative holonomic, and $(n+r)$-Cohen-Macaulay. The process finishes at the case $r=0$, in which case there is nothing to prove anymore.
\end{proof}

\begin{rmk}
A result similar to Proposition \ref{prop:maincm} is proved by a different method in \cite[Appendix B]{Bath2} for $\sD_X[\bs]\bf^\bs/\sD_X[\bs]\bf^{\bs+\mathbf 1}$ when $\bf$ is a reduced free hyperplane arrangement.
\end{rmk}

\subsection{} We consider now  the left $\sA$-module
$$M=\sD_X[\bs]\bf^\bs/\sD_X[\bs]\bf^{\bs+\mathbf 1}.$$
In this case, the annihilator $B_M$ is the Bernstein-Sato ideal $B_F$, since $M$ is a cyclic $\sA$-module generated by the class of $\bf^\bs$ in $M$. 

It is well-known that the zero locus $Z(B_F)$ in $\bC^r$ has dimension $r-1$. Indeed, since $B_F$ is the intersection of the local Bernstein-Sato ideals, by restricting attention to the neighborhood of a smooth point of the zero locus of $\prod_{i=1}^rf_i$, one reduces the assertion to the case when $f_i=x_1^{a_i}$ for some $a_i\in\bN$ for all $i=1,\ldots r$ with $\ba=(a_1,\ldots, a_r)\neq (0,\ldots ,0)$. In this case, the Bernstein-Sato ideal is principal, generated by $\prod_{j=1}^{|\ba|}(\ba\cdot\bs + j)$ with $|\ba|=a_1+\ldots+a_r$. 

In addition, it is known that every top-dimensional irreducible component of $Z(B_F)$ is a hyperplane in $\bC^r$ defined over $\bQ$ by \cite{Sab}.

We will use the following result of Maisonobe, which also holds in the local analytic case, cf. \ref{subAC}:

\begin{theorem}\label{thmMais} (Maisonobe) The $\sA$-module $M$ is relative holonomic over $\bC[\bs]$, has grade number $j(M)=n+1$ over $\sA$, and $\dim \Chr(M)=n+r-1$. Every irreducible component of $Z(B_F)$ of codimension $>1$ can be translated by an element of $\bZ^r$ into a component of codimension one.
\end{theorem}
\begin{proof} In \cite[R\'esultat 3]{M} it is shown that $\Chr(M)=\cup_{i\in I}\Lambda_i\times Z_i$ for some finite set $I$ with $\Lambda_i\subset T^*X$ conic Lagrangian, $Z_i\subset \bC^r$ algebraic closed subset of dimension $\le r-1$. Thus $M$ is relative holonomic over $\bC[\bs]$. Lemma \ref{lm:Rsupp} shows that $Z(B_F)=\cup_{i\in I}Z_i$, cf. also the remark after \cite[R\'esultat 2]{M}. Since $\dim Z(B_F)=r-1$, it follows that $\dim \Chr(M)=n+r-1$, and hence $j(M)=n+1$ by Lemma \ref{thm:relCh} (1). The last claim is contained in the statement of \cite[R\'esultat 3]{M}.
\end{proof}

We next observe that over an open subset of $\bC^r$,  $M$ behaves particularly nice:

\begin{lemma}\label{lemGENE}
There exists an open affine subset $V=\Spec R\subset\bC^r$ such that the intersection of   $V$ with each irreducible component of codimension one of $Z(B_F)$ is not empty, and
the  module $M\otimes_{\bC[\bs]}R$ is relative holonomic over $R$ and $(n+1)$-Cohen-Macaulay over $\sA_R$.
\end{lemma}
\begin{proof} Since $M$ is relative holonomic over $\bC[\bs]$, and since good filtrations localize by Remark \ref{rmkF}, it follows that $M\otimes_{\bC[\bs]}R$ is relative holonomic over $R$, if $\Spec R$ is a non-empty open subset of $\bC^r$.
 
 Since $j(M)=n+1$, 
\[\Ext^k_\sA(M,\sA)=0 \textup{ for } k<n+1.\]
By Auslander regularity of $\sA$, if $\Ext^k_\sA(M,\sA))\not=0$ for $k\ge n+1$, then
\[j(\Ext^k_\sA(M,\sA))\ge k.\]
Note that since $\text{gl.dim}(\sA)$ is finite, there are only finitely many $k$ with $\Ext^k_\sA(M,\sA)\ne 0$.
By Lemma \ref{lm:relhol} (2),  if $\Ext^k_\sA(M,\sA)\not=0$, then $\Ext^k_\sA(M,\sA))$ is relative holonomic and 
\[\Chr(\Ext^k_\sA(M,\sA)))\subset \Chr(M).\]
By Lemma \ref{thm:relCh} (1), when $k> n+1$, 
\begin{equation}\label{eqdq}\dim(\Chr(\Ext^k_\sA(M,\sA)))< n+r-1.\end{equation}

By relative holonomicity, the irreducible components of $\Chr(M)$ are $\Lambda_i\times Z_i$ with $i$ in some finite index set $I$, $\Lambda_i\subset T^*X$ irreducible conic Lagrangian, and $Z_i$ irreducible closed in $\bC^r$. Then the irreducible components of  $\Chr(\Ext^k_\sA(M,\sA))$ are $\Lambda_i\times Z_i'$ with $i$ in some subset $J\subset I$, and $Z_i'$ irreducible closed in $Z_i$. By Lemma \ref{lm:Rsupp} applied to $M$ and $\Ext^k_\sA(M,\sA)$, respectively, we have that $Z(B_F)=\cup_{i\in I}Z_i$, and the support in $\bC^r$ of $\Ext^k_\sA(M,\sA)$ is $\cup_{i\in J}Z_i'$. Then $\dim Z(B_F)=r-1$, and $\dim Z_i'< r-1$ for each $k>n+1$ by (\ref{eqdq}). Therefore the support in $\bC^r$ of $\Ext^k_\sA(M,\sA)$ is a proper algebraic subset of $Z(B_F)$ not containing any top-dimensional component of $Z(B_F)$ if $k>n+1$. Choose $V=\Spec R$ to be an open affine subset of $\bC^r$ away from these proper subsets of $Z(B_F)$ for all $k>n+1$. Then for any good filtration we have
$$
(\gr \Ext^k_\sA(M,\sA))\otimes_{\bC[\bs]}R =0
$$
for all $k>n+1$. Since $R$ is the localization of $\bC[\bs]$ with respect to some multiplicatively closed subset $S$, and since good filtrations localize, cf. Remark \ref{rmkF}, we have 
$$\gr (S^{-1} \Ext^k_\sA(M,\sA))=0,$$ and so $$S^{-1} \Ext^k_\sA(M,\sA)=0.$$
Since $S$ is also a multiplicatively closed subset of $\sA$, in the center of $\sA$, and $M$ is finitely generated over the noetherian ring $\sA$, the $\Ext$ module localizes 
$$
0=S^{-1} \Ext^k_\sA(M,\sA) =  \Ext^k_{S^{-1}\sA}(S^{-1}M,S^{-1}\sA),
$$
cf. \cite[Lemma 3.3.8]{HAWei} and the proof of \cite[Proposition 3.3.10]{HAWei}, where one identifies the localization functor $S^{-1}(\_)$ on $\sA$-modules with the flat extension $(\_)\otimes_{\sA}\sA_R=(\_)\otimes_{\bC[\bs]}R$.
Thus $S^{-1}M=M\otimes_{\bC[\bs]}R $ is $(n+1)$-Cohen-Macaulay over $S^{-1}\sA=\sA_R$. 
\end{proof}

Now Lemma \ref{lemGENE} and Proposition \ref{prop:maincm} immediately imply:
\begin{theorem}\label{thmGCz} For every irreducible component $H$ of codimension one of $Z(B_F)$ and for every general point $\bal$ on $H$, 
\[M\otimes_{\bC[\bs]}\bC_\bal\not=0.\]
\end{theorem}

\subsection{Analytic case.}\label{subAC} Theorem \ref{thmGCz} holds also in the local analytic case. We indicate now the necessary changes in the arguments. The smooth affine algebraic variety $X$ is replaced by the germ $(X,x)$ of a complex manifold of dimension $n$. The rings $R$ stay as before and we let $Y$ denote the complex manifold underlying the smooth affine complex algebraic variety $\Spec (R)$. The rings and modules from the algebraic case $\sD_X$, $\sA_R=\sD_X\otimes_\bC R$, $N$, etc., have natural analytic versions as sheaves on the complex manifold $X$, but their role from the previous arguments will be played by the stalks of these sheaves, $\sD_{X,x}$, $\sA_{R,x}=\sD_{X,x}\otimes_\bC R$, $N_x$, etc.  The role of $\Chr(N)$ from the algebraic case will be played by $\Chr(N)\cap\pi^{-1}(\Omega\times Y)$, for a very small open ball $\Omega$ in $X$ centered at $x$. Recall that for a coherent sheaf of $\sA_R$-modules $N$ on the complex manifold $X$, the relative characteristic variety $\Chr(N)$ is the analytic subspace of $T^*X\times Y$ defined as the zero locus of the radical of the annihilator of $N$ in $\sA_R$.  With these changes, all the statements in this section hold in the local analytic case as well. 

There are however a few special issues arising in this case, since (partial) analytifications of $\sA_R$ and $N$ are needed in order for the module theory as in the Appendix to capture the analytic structure of $\Chr(N)$. For a sheaf of $\sO_X\otimes_\bC R$-modules $L$ on the complex manifold $X$, one defines the (partial) analytification
$$\widetilde{L}=\sO_{X\times Y} \otimes_{p^{-1}(\sO_X\otimes_\bC R)}p^{-1}(L),$$
a  sheaf of $\sO_{X\times Y}$-modules, where $p:X\times Y\ra X$ is the first projection. Thus $\widetilde{\sA_R}$ is the sheaf of relative differential operators $\sD_{X\times Y/Y}$, locally isomorphic to $\sO_{X\times Y}[\partial_1,\ldots,\partial_n]$. The analytification of the filtration on $\sA_R$ is the natural filtration on $\widetilde{\sA_R}$, and $\gr \sA_R$ is locally isomorphic to $\sO_{X\times Y}[\xi_1,\ldots,\xi_n]$, a sheaf of subrings of $\sO_{T^*X\times Y}$, where $\xi_i$ are coordinates of the fibers of the natural projection $\pi:T^*X\times Y\ra X\times Y$. If $N$ is a coherent sheaf of $\sA_R$-modules, then $\widetilde{N}$ is a coherent sheaf of $\widetilde{\sA_R}$-modules. Since $\widetilde{(\_)}$ is an exact functor, it is is compatible with good filtrations, $\gr \widetilde{N}=\widetilde{\gr N}$, the annihilator in $\gr\widetilde{\sA_R}$ of $\gr \widetilde{N}$ is the analytification of the annihilator of $\gr N$ in $\sA_R$, and the radical $J(\widetilde{N})$ of the former is the analytification $\widetilde{J(N)}$ of the radical of the latter. Then $\Chr(N)$ is the analytic subspace of $T^*X\times Y$ defined by the ideal generated by $J(\widetilde{N})$ in $\sO_{T^*X\times Y}$, the full analytification, cf. \cite[I.6.21]{Bj}.

Note that there is a natural isomorphism of $\bC$-algebras $$\gr\sA_{R,x}\simeq\bC\{x_1,\ldots,x_n\}[\xi_1,\ldots,\xi_n]\otimes_\bC R$$
after choosing local coordinates $x_1,\ldots, x_n$ on $X$ at $x$. This ring is a regular commutative integral domain of  dimension $2n+\dim(R)$. Thus all the results in the Appendix apply to this ring, except Proposition \ref{propCo} (ii).  Indeed, $\gr\sA_{R,x}$   has maximal ideals of height less than $\dim(\gr\sA_{R,x})$. (For example, the ideal $(1-x\xi)$ of $\bC\{x\}[\xi]$ is maximal of height 1.) On the other hand, our modules are special: $\gr N_x$ is a graded module if $\gr\sA_{R,x}$ is given the natural grading in the coordinates $\xi_1,\ldots,\xi_n$. The exact functor $\widetilde{(\_)}$ is also faithful on the category of coherent graded $\gr\sA_R$-modules:

\begin{prop}\label{propMFt} (Maisonobe \cite[Lemme 1]{M})
If $M$ is a coherent $\gr\sA_R$-module and $x\in X$, then $M_x=0$ if and only if there exists an open neighborhood $\Omega$ of $x$ in $X$ such that $\widetilde{M}|_{\Omega\times Y}=0$.
\end{prop}
Thus one obtains, cf. \cite[Proposition 2]{M}: for a small enough $\Omega$,  
$$
j_{\gr\sA_{R,x}}(\gr N_x) = \inf_{(x',y)\in\Omega\times Y} j_{(\gr \widetilde{\sA_R})_{(x',y)}} ((\gr\widetilde{N})_{(x',y)}).
$$
The  stalks  $(\gr\widetilde{N})_{(x',y)}$  determine  the local analytic structure at $(x',0,y)$ of the conical set $\Chr(N)$,  since the extension functor from the category of graded  coherent sheaves over $\gr \widetilde{\sA_R}$ into the category of coherent sheaves over $\sO_{T^*X\times Y}$ is also  faithful besides being exact, by the Nullstellensatz for conical analytic sets, cf. \cite[Remark I.1.6.8]{Bj}. In particular, there is a 1-1 correspondence between conical  analytic sets in $T^*X\times Y$ and radical graded coherent ideals in $\gr\widetilde{\sA_R}$. Therefore the ring $(\gr \widetilde{\sA_R})_{(x',y)}$ and the module $(\gr\widetilde{N})_{(x',y)}$ can be replaced by  their localization at the unique graded maximal ideal (cf. \cite[1.5]{BH}) and in this context Proposition \ref{propCo} (ii) does apply. A consequence is that Lemma \ref{thm:relCh} (1) holds indeed with the changes we have mentioned: for a small neighborhood $\Omega$ of $x$,
$$
j_{\sA_{R,x}} (N_x) + \dim (\Chr(N)\cap \pi^{-1}(\Omega\times Y))=2n+\dim(R).
$$
This is \cite[Proposition 2, Th\'eor\`eme 1]{M}, where $R=\bC[\bs]$ but the proof applies in general, and we used semicontinuity of the dimension function \cite[p.94]{GR} to rephrase the statement slightly.

Next, in keeping up with the changes indicated, the condition ``regular holonomic"   will be replaced by the  condition that a coherent module $N$ over $\sA_R$ is {\it regular holonomic  at $x$}, that is, there exists a neighborhood $\Omega$ of $x$ such that $\Chr(N)\cap\pi^{-1} (\Omega\times Y)$ is as in Definition \ref{defRHL}. 

The condition ``$j$-Cohen-Macaulay"  will be replaced by the condition that $N$ is {\it $j$-Cohen-Macaulay at $x$}, that is, $N_x$ is $j$-Cohen-Macaulay. This is equivalent to $N$ being $j$-Cohen-Macaulay on some neighborhood $\Omega$ of $x$, that is, $j$-Cohen-Macaulay  at all points $x'$ in $\Omega\cap\supp (N)$. Note that the support of $N$ is an analytic subset of $X$ by Proposition \ref{propMFt}, since the support of $\widetilde{N}$ is an analytic subset of $X\times Y$ by the conical property of $\Chr(N)$. Moreover, $N$ is $j$-Cohen-Macaulay on $\Omega$ if and only if one of the following two equivalent conditions hold for $k\neq j$: $\mathscr{E}xt^k_{\sA_R}(N,\sA_R)|_\Omega=0$;  $\mathscr{E}xt^k_{\sA_R}(N,\sA_R)_{x'}=0$ for all $x'\in \Omega$.
Also, $N$ is $j$-Cohen-Macaulay at $x$ if and only if $\widetilde{N}$ is $j$-Cohen-Macaulay on $\Omega\times Y$ for some $\Omega\ni x$, by Proposition \ref{propMFt}. This implies, by applying Proposition \ref{propCo} in the context mentioned above, that  Remark \ref{remdefnImp} holds in the local analytic case; in particular, if $N$ is $j$-Cohen-Macaulay at $x$, then $\Chr(N)\cap\pi^{-1}(\Omega\times Y)$ is equidimensional of codimension $j$.

With the changes we have indicated, the rest of the arguments remain as before, and all statements in this section are true in this case.

\subsection{Proof of Theorem \ref{thmconj}.}  By Theorem \ref{thmGCz} and  Proposition \ref{propBS}, the image under $\Exp$ of a non-empty open subset of each irreducible component of codimension one of $Z(B_F)$ lies in $\cS(F)$. By the description of $Z(B_F)$ from Theorem \ref{thmMais} and the paragraphs preceding it, it follows that $\Exp(Z(B_F))$ is included in $\cS(F)$. $\hfill\Box$

\section{Appendix}\label{secApp}

We recall some facts for not-necessarily commutative rings from \cite[A.III and A.IV]{Bj} that we use in the proof of the main theorem. 

\subsection{}\label{subNoe} 

Let $A$ be a ring, by which we mean an associative ring with a unit element. Let $\Mod_{f}(A)$ be the abelian category of finitely generated left $A$-modules.

We say that $A$ is a {\it positively filtered ring} if $A$ is endowed with a  $\bZ$-indexed increasing exhaustive filtration $\{F_iA\}_{i\in\bZ}$ of additive subgroups  such that $F_iA\cdot F_jA\subset F_{i+j}A$ for all $i, j$ in $\bZ$, and $F_{-1}A=0$. The associated graded object $\grF  A=\oplus_i (F_iA/F_{i-1}A)$ has a natural ring structure. When we do not need to specify the filtration, we write $\gr A$ for $\grF A$.

If $A$ is a positively filtered ring such that $\gr A$ is noetherian, then $A$ is noetherian, \cite[A.III 1.27]{Bj}. Here, noetherian means both left and right noetherian.

\subsection{}\label{subFlt} Let $A$ be a noetherian ring, positively filtered.  A {\it good filtration} on $M\in\Mod_f(A)$ is an increasing exhaustive filtration $F_\bullet M$ of additive subgroups such that $F_iA\cdot F_jM\subset F_{i+j}M$ for all $i, j$ in $\bZ$, and such that its associated graded object $\gr M$ is a finitely generated graded module over $\gr A$, cf. \cite[A.III 1.29]{Bj}.

\begin{prop}\label{lem:relCh} (\cite[A.III 3.20 - 3.23]{Bj})
Let $A$ be a noetherian ring, positively filtered. 
\begin{enumerate}
\item  Let $M$ be in $\Mod_f(A)$ with a good filtration. Then the radical of the annihilator ideal in $\gr A$
$$
J(M):=\sqrt{\Ann_{\gr A}(\gr M)}
$$
and the multiplicities $m_\mathfrak{p}(M)$ of $\gr  M$ at minimal primes $\mathfrak{p}$ of $J(M)$ do not depend on the choice of a good filtration.
\item If $$0\to M'\to M\to M''\to 0$$
is an exact sequence in $\Mod_f(A)$ then
$$
J(M)=J(M')\cap J(M'')
$$ and  if $\mathfrak{p}$ is a minimal prime of $J(M)$ then \[m_\mathfrak{p}(M)=m_\mathfrak{p}( M')+m_\mathfrak{p}( M'').\]
\end{enumerate}
\end{prop}

Note that the last assertion is equivalent to the existence of a $\bZ$-valued additive map $m_{\mathfrak p}$ on the Grothendieck group generated by the finitely generated modules $N$ over $\gr A$ with $J(M)\subset \sqrt{\Ann_{\gr A} N}$, as it is phrased in {\it loc. cit.}

\begin{prop}\label{propExGr} (\cite[A.IV 4.5]{Bj}) Let $A$ be a noetherian ring, positively filtered. Let $M$ be in $\Mod_f(A)$ with a good filtration. For every $k\ge 0$, there exists a good filtration on the right $A$-module $\Ext^k_A(M,A)$ such that $\gr(\Ext^k_A(M,A))$ is a subquotient of $\Ext^k_{\gr A}(\gr M,\gr A)$.
\end{prop}

\subsection{}\label{subGrade} Let $A$ be a noetherian ring. The smallest $k\ge 0$ for which every $M$ in $\Mod_f(A)$ has a projective resolution of length $\le k$ is called the {\it homological dimension} of $A$ and it is denoted by
$
\textup{gl.dim}(A).
$

\begin{defn}
For a nonzero $M$ in $\Mod_f(A)$, the smallest integer $k\ge 0$ such that $\Ext^k_A(M,A)\ne 0$ is  denoted $$j_A(M)$$
and it is called the {\it grade number} of $M$. If $M=0$ the grade number is taken to be $\infty$.
\end{defn}

The ring $A$ is {\it Auslander regular} if it has finite homological dimension and, for every $M$ in $\Mod_f(A)$, every $k\ge 0$, and every nonzero right submodule $N$ of $\Ext^k_A(M,A)$, one has $j_A(N)\ge k$. This implies the similar condition phrased for right $A$-modules $M$, see \cite[A.IV 1.10]{Bj} and the comment thereafter.

\begin{theorem}\label{thrmGAAR}
(\cite[A.IV 5.1]{Bj}) If $A$ is a positively filtered ring such that $\gr A$ is a regular commutative ring, then $A$ is an Auslander regular ring.
\end{theorem}

\begin{prop}\label{eqGldim} (\cite[A.IV 1.11]{Bj}) Let $A$ be an Auslander regular ring. Then
\[\textup{gl.dim}(A)=\sup \{j_A(M)\mid 0\neq M\in \Mod_f(A)\}.\]
\end{prop}

\begin{definition}\label{defjpure} 
A nonzero module $M$  in $\Mod_f(A)$ is {\it $j$-pure} (or simply, {\it pure}) if $j_A(N)=j_A(M)=j$ for every nonzero submodule $N$. 
\end{definition}

\begin{lemma}\label{lemEpure}(\cite[A.IV 2.6]{Bj}) Let $A$ be an Auslander regular ring, $M$ nonzero in $\Mod_f(A)$, and $j=j_A(M)$. Then:
\begin{itemize}
\item[(1)]  $\Ext^j_A(M,A)$ is a $j$-pure right $A$-module;
\item[(2)] $M$ is pure if and only if $\Ext^k_A(\Ext^k_A(M,A),A)=0$ for every $k\ne j$.
\end{itemize}
\end{lemma}

\subsection{} We assume now that $A$ is a positively filtered ring such that $\gr A$ is a regular commutative  ring. Then $A$ is also Auslander regular by Theorem \ref{thrmGAAR}. Moreover, with these assumptions one has the following two results.

\begin{prop}\label{propPGF} (\cite[A.IV 4.10 and 4.11]{Bj}) If $M$ in $\Mod_f(A)$ is $j$-pure, there exists a good filtration on $M$ such that $\gr M$ is a $j$-pure $\gr A$-module.
\end{prop}

\begin{prop}\label{propJgr}(\cite[A.IV 4.15]{Bj}) For any $M$ in $\Mod_f(A)$ and any good filtration on $M$, $$j_A(M)=j_{\gr  A}(\gr M).$$
\end{prop}
 
\subsection{} Lastly, we consider a regular commutative ring $A$.  Then ${\rm gl.dim.}(A)=\sup\{{\rm gl.dim.}(A_\mathfrak{m})\mid \mathfrak{m}\subset A\text{ maximal ideal}\,\}$, cf. \cite[Ch. 2, 5.20]{Bjo}. We let $\dim(A)$ denote the Krull dimension. For a module $M\in \Mod_f(A)$,  $\dim_A(M)$ denotes $\dim(A/ \Ann_A(M))$.  If $A$ is a regular local commutative ring, then $\dim(A)={\rm gl.dim.}(A)$, cf. \cite[A.IV 3.5]{Bj}.

\begin{prop}\label{propCo} 
Let $A$ be a regular commutative ring and $M$ nonzero in $\Mod_f(A)$. Then:
\begin{itemize}
\item[$(i)$]  (\cite[A.IV 3.4]{Bj})  $A$ is Auslander regular;
\item[$(ii)$] (\cite[Ch. 2, Thm. 7.1]{Bjo}) if  $\dim(A_\mathfrak{m})=m$ for every maximal ideal $\mathfrak{m}$ of $A$,
$$j_A(M)+\dim_A(M) = m\, ;$$
\item[$(iii)$] (\cite[A.IV 3.7 and 3.8]{Bj}) $M$ is a pure $A$-module if and only if every associated prime of $M$ is a minimal prime of $M$ and $j_A(M)=\dim (A_{\mathfrak{p}})$ for every minimal prime $\mathfrak{p}$ of $M$.
\end{itemize}
\end{prop}

\end{document}